\documentclass[a4paper, 12pt]{article}

\usepackage[margin=2cm]{geometry}
\usepackage{hyperref,xcolor}
\usepackage{amsthm, amssymb, amsmath}
\usepackage{enumerate}
\usepackage{algorithmic}
\usepackage{authblk}

\hypersetup{
    colorlinks=false,
	linkcolor=red,
	citecolor=red,
    pdfborder={0 0 0},
}

\title{Separation dimension of sparse graphs}

\author[1]{Manu~Basavaraju}
\author[2]{L.~Sunil~Chandran}
\author[3]{Rogers~Mathew\footnote{Supported by VATAT Post-doctoral Fellowship, Council of Higher Education, Israel.}}
\author[3]{Deepak~Rajendraprasad\footnote{Supported by VATAT Post-doctoral Fellowship, Council of Higher Education, Israel.}}

\affil[1]{
	University of Bergen, Postboks 7800, NO-5020 Bergen.\authorcr
	\texttt{manu.basavaraju@ii.uib.no}
}
\affil[2]{
	Department of Computer Science and Automation, \authorcr 
	Indian Institute of Science, Bangalore, India - 560012. \authorcr
	\texttt{sunil@csa.iisc.ernet.in}
}
\affil[3]
{
	The Caesarea Rothschild Institute, Department of Computer Science, \authorcr
	University of Haifa, 31095, Haifa, Israel. \authorcr
	\texttt{\{rogersmathew, deepakmail\}@gmail.com}
}

\theoremstyle{definition}
\newtheorem{definition}{Definition}

\theoremstyle{plain}
\newtheorem{theorem}{Theorem}
\newtheorem{lemma}[theorem]{Lemma}

\theoremstyle{remark}

\newtheoremstyle{plainitshape}
  {}
  {}
  {\itshape}
  {}
  {\itshape}
  {.}
  {0.5em}
  {}
\theoremstyle{plainitshape}
\newtheorem{claim}{Claim}[theorem]
\newtheorem{property}{Property}

\newtheoremstyle{cases}
  {}
  {}
  {}
  {}
  {}
  {\newline}
  {0.5em}
  {{\itshape \thmname{#1}} \thmnumber{#2} ({\itshape\thmnote{#3}}).\medskip}
\theoremstyle{cases}
\newtheorem{case}{Case}

\newtheoremstyle{constructions}
  {}
  {}
  {}
  {}
  {}
  {}
  {0.5em}
  {{\itshape \thmname{#1}} \thmnumber{#2}\medskip}
\theoremstyle{constructions}
\newtheorem{construction}{Construction}[theorem]

\newcommand{\boxli}{\pi}

\newcommand{\idim}{\operatorname{dim}}

\newcommand{\leftend}{l}
\newcommand{\rightend}{r}

\newcommand{\floor}[1]{\left\lfloor #1 \right\rfloor}

\newcommand{\ilog}{\operatorname{log^{\star}}}
\newcommand{\order}[1]{O\left( #1 \right)}
\newcommand{\orderatleast}[1]{\Omega\left( #1 \right)}
\newcommand{\orderexactly}[1]{\Theta\left( #1 \right)}

\def\into{\rightarrow}
\def\N{\mathbb{N}}
\def\R{\mathbb{R}}

\def\F{\mathcal{F}}
\def\E{\mathcal{E}}

\begin{document}
\maketitle

\begin{abstract}
The {\em separation dimension} of a graph $G$ is the smallest natural number $k$ for which the vertices of $G$ can be embedded in $\R^k$ such that any pair of disjoint edges in $G$ can be separated by a hyperplane normal to one of the axes. Equivalently, it is the smallest possible cardinality of a family $\F$ of permutations of the vertices of $G$ such that for any two disjoint edges of $G$, there exists at least one permutation in $\F$ in which all the vertices in one edge precede those in the other. In general, the maximum separation dimension of a graph on $n$ vertices is $\orderexactly{\log n}$. In this article, we focus on sparse graphs and show that the maximum separation dimension of a $k$-degenerate graph on $n$ vertices is $\order{k \log\log n}$ and that there exists a family of $2$-degenerate graphs with separation dimension $\orderatleast{\log\log n}$. We also show that the separation dimension of the graph $G^{1/2}$ obtained by subdividing once every edge of another graph $G$ is at most $(1 + o(1)) \log\log \chi(G)$ where $\chi(G)$ is the chromatic number of the original graph.

\vspace{1ex}
\end{abstract}


\section{Introduction}

Let $\sigma:U \into [n]$ be a permutation of  elements of an $n$-set $U$. For two disjoint subsets $A,B$ of $U$, we say  $A \prec_{\sigma} B$ when every element of $A$ precedes every element of $B$ in $\sigma$, i.e., $\sigma(a) < \sigma(b), ~\forall (a,b) \in A \times B$. 
We say that $\sigma$ {\em separates} $A$ and $B$ if either $A \prec_{\sigma} B$ or $B \prec_{\sigma} A$. We use $a \prec_{\sigma} b$ to denote $\{a\} \prec_{\sigma} \{b\}$. For two subsets $A, B$ of $U$, we say $A \preceq_{\sigma} B$ when $A \setminus B \prec_{\sigma} A\cap B \prec_{\sigma} B \setminus A$. 

Families of permutations which satisfy some type of ``separation'' properties have been long studied in combinatorics. One of the early examples of it is seen in the work of Ben Dushnik in 1950 where he introduced the notion of \emph{$k$-suitability} \cite{dushnik}. A family $\F$ of permutations of $[n]$ is \emph{$k$-suitable} if, for every $k$-set $A \subseteq [n]$ and for every $a \in A$, there exists a $\sigma \in \mathcal{F}$ such that $A \preceq_{\sigma} \{a\}$. Let $N(n,k)$ denote the cardinality of a smallest family of permutations that is $k$-suitable for $[n]$. In 1971, Spencer \cite{scramble} proved that $\log \log n \leq N(n,3) \leq N(n,k) \leq k2^k\log \log n$. He also showed that $N(n,3) < \log \log n + \frac{1}{2} \log \log \log n + \log (\sqrt{2}\boxli) + o(1)$. Fishburn and Trotter, in 1992, defined the \emph{dimension} of a hypergraph on the vertex set $[n]$ to be the minimum size of a family $\F$ of permutations of $[n]$ such that every edge of the hypergraph is an intersection of \emph{initial segments} of $\F$ \cite{FishburnTrotter1992}. It is easy to see that an edge $e$ is an intersection of initial segments of $\F$ if and only if for every $v \in [n] \setminus e$, there exists a permutation $\sigma \in \F$ such that $e \prec_{\sigma} \{v\}$. F\"{u}redi, in 1996, studied the notion of \emph{$3$-mixing} family of permutations \cite{Furedi1996}. A family $\F$ of permutations of $[n]$ is called $3$-mixing if for every $3$-set $\{a, b, c\} \subseteq [n]$ and a designated element $a$ in that set, one of the permutations in $\F$ places the element $a$ between $b$ and $c$. It is clear that $a$ is between $b$ and $c$ in a permutation $\sigma$ if and only if $\{a,b\} \preceq_{\sigma} \{a,c\}$ or $\{a.c\} \preceq_{\sigma} \{a,b\}$. Such families of permutations with small sizes have found applications in showing upper bounds for many combinatorial parameters like poset dimension \cite{kierstead1996order}, product dimension \cite{FurediPrague}, boxicity \cite{RogSunSiv} etc.  

This paper is a part of our broad investigation\footnote{Most of our initial results on this topic are available as a preprint in arXiv \cite{basavaraju2012pairwise}. This paper is a subset of the same. A disjoint subset of the results there along with some new ones have been submitted to WG 2014 and is currently under review.}  on a similar class of permutations which we make precise next.

\begin{definition}
\label{definitionPairwiseSuitable}
A family $\F$ of permutations of $V(H)$ is \emph{pairwise suitable} for a hypergraph $H$ if, for every two disjoint edges $e,f \in E(H)$, there exists a permutation $\sigma \in \F$  which separates $e$ and $f$. The cardinality of a smallest family of permutations that is pairwise suitable for $H$ is called the {\em separation dimension} of $H$ and is denoted by $\boxli(H)$. 
\end{definition}

A family $\F = \{\sigma_1, \ldots, \sigma_k\}$ of permutations of a set $V$ can be seen as an embedding of $V$ into $\R^k$ with the $i$-th coordinate of $v \in V$ being the rank of $v$ in the $\sigma_i$. Similarly, given any embedding of $V$ in $\R^k$, we can construct $k$ permutations by projecting the points onto each of the $k$ axes and then reading them along the axis, breaking the ties arbitrarily. From this, it is easy to see that $\boxli(H)$ is the smallest natural number $k$ so that the vertices of $H$ can be embedded into $\R^k$ such that any two disjoint edges of $H$ can be separated by a hyperplane normal to one of the axes. This motivates us to call such an embedding a {\em separating embedding} of $H$ and $\boxli(H)$ the {\em separation dimension} of $H$. 

A major motivation for us to study this notion of separation is its interesting connection with a certain well studied geometric representation of graphs. The \emph{boxicity} of a graph $G$ is the minimum natural number $k$ for which $G$ can be represented as an intersection graph of axis-parallel boxes in  $\R^k$.  The separation dimension of a hypergraph $H$ is equal to the boxicity of the intersection graph of the edge set of $H$, i.e., the line graph of $H$ \cite{basavaraju2012pairwise}.


\section{Separating sparse graphs}

It is known that $\boxli(K_n) \in \orderexactly{\log n}$, where $K_n$ denotes the complete graph on $n$ vertices \cite{basavaraju2012pairwise}. It is easy to see that separation dimension is a monotone property, i.e., $\boxli(G') \leq \boxli(G)$ if $G'$ is a subgraph of $G$. So it is interesting to check whether the separation dimension of sparse graph families can be much lower than $\log n$. Here, by a {\em sparse family of graphs}, we mean a family of graphs with $m \in O(n)$ where $m$ and $n$ denote, respectively, the number of edges and vertices in the graph. One way to ensure sparsity of a family of graphs is to demand that the maximum degree of the graphs be bounded. We know that for any graph with maximum degree at most $\Delta$, its separation dimension is at most $2^{9 \ilog \Delta} \Delta$ \cite{basavaraju2012pairwise}. But not all families of sparse graphs have bounded maximum degree - trees for instance. Sparsity of a graph family, as we consider it here, is equivalent to the restriction that the graphs in the family have a bounded average degree. But globally sparse graphs with a small dense subgraph can have large separation dimension. For example if we consider an $n$ vertex graph which is a disjoint union of a complete graph on $\floor{\sqrt{n}}$ vertices and remaining isolated vertices, it has at most $n/2$ edges, but has a separation dimension in $\orderatleast{\log n}$ due to the clique (monotonicity). Hence we see that sparsity is needed across all subgraphs in order to hope for a better upper bound for separation dimension. One common way of ensuring local sparsity of a graph family is to demand that the degeneracy (cf. Definition \ref{definitionDegeneracy}) of the graphs in the family be bounded. 

\begin{definition}
\label{definitionDegeneracy}
For a non-negative integer $k$, a graph $G$ is \emph{$k$-degenerate} if the vertices of $G$ can be enumerated in such a way that every vertex is succeeded by at most $k$ of its neighbours. The least number $k$ such that $G$ is $k$-degenerate is called the \emph{degeneracy} of $G$ and any such enumeration is referred to as a \emph{degeneracy order} of $V(G)$. 
\end{definition}

For example, trees and forests are 1-degenerate and planar graphs are 5-degenerate. Series-parallel graphs, outerplanar graphs, non-regular cubic graphs, circle graphs of girth at least 5 etc. are 2-degenerate.  It is easy to verify that if the maximum average degree over all subgraphs of a graph $G$ is $d$, then $G$ is $\floor{d}$-degenerate. A $\floor{d}$-degeneracy order of $G$ can be obtained by recursively picking out a minimum degree vertex from $G$. It is also easy to see that any subgraph of a $k$-degenerate graph has maximum degree at most $2k$.

In this paper we establish the following upper bound on separation dimension of $k$-degenerate graphs and there by give an affirmative answer to our question under a restricted but necessary condition of sparsity. 
 
\begin{theorem}
\label{theoremBoxliDegeneracy}
For a $k$-degenerate graph $G$ on $n$ vertices, $\boxli(G) \in O(k \log \log n)$.  
\end{theorem}

We prove this by decomposing $G$ into $2k$ star forests and using $3$-suitable permutations of the stars in every forest and the leaves in every such star simultaneously. The proof is given in Appendix \ref{sectionDegeneracy}. We show that the $\log \log n$ factor in Theorem \ref{theoremBoxliDegeneracy} cannot be improved in general by estimating the exact order of the separation dimension of a fully subdivided clique.

\begin{definition}
A graph $G'$ is called a {\em subdivision} of a graph $G$ if $G'$ is obtained from $G$ by replacing a subset of edges of $G$ with independent paths between their ends such that none of these new paths has an inner vertex on another path or in $G$. A subdivision of $G$ where every edge of $G$ is replaced by a $k$-length path is denoted as $G^{1/k}$. The graph $G^{1/2}$ is called {\em fully subdivided} $G$.
\end{definition}

It is easy to see that $G^{1/2}$ is a $2$-degenerate graph for any graph $G$. A $2$-degeneracy order can be obtained by picking out the vertices introduced by the subdivision first. 
\begin{theorem}
\label{theoremKnHalf}
Let $K_n^{1/2}$ denote the graph obtained by fully subdividing $K_n$. Then, 
$$ \frac{1}{2} \floor{\log\log(n-1)} \leq \boxli(K_n^{1/2}) \leq (1 + o(1))\log\log (n-1).$$ 
\end{theorem}

We establish the lower bound, quite laboriously, by using Erd\H{os}-Szekeres Theorem to extract a large enough set of vertices of the underlying $K_n$ that are ordered essentially the same by every permutation in the selected family and then showing that separating the edges incident on those vertices can be modelled as a problem of finding a realiser for a canonical open interval order of same size. The details are given in Appendix \ref{sectionSubdividedClique}. The upper bound follows from the next result.

Prompted by the above two results, we investigate deeper the separation dimension of fully subdivided graphs and establish the following.

\begin{theorem}
\label{theoremSubdivisionChromaticNumber}
For a graph $G$ with chromatic number $\chi(G)$, 
$$ 
	\boxli(G^{1/2}) \leq 
		\log\log (\chi(G)-1) + 
		\left( \frac{1}{2} + o(1) \right) \log\log\log (\chi(G)-1) + 2.
$$
\end{theorem}

We do this by associating with every graph $G$ an interval order whose dimension (cf. Definition \ref{definitionIntervalDimension} in Appendix \ref{sectionSubdivision}) is at least $\boxli(G^{1/2})$ and whose height is less than the chromatic number of $G$ and then using a result on the dimension of interval orders due to F\"{u}redi, Hajnal, R\"{o}dl and Trotter \cite{furedi1991interval}. The details are given in Appendix \ref{sectionSubdivision}. The tightness, up to a factor of $2$, of the above bound follows from the previous result that $\boxli(K_n^{1/2}) \geq \frac{1}{2} \floor{\log\log (n-1)}$.



\bibliographystyle{plain}
\bibliography{/home/deepak/Research/LaTeX/deepak.bib}

\clearpage
\appendix

\section{Notational note}
\label{sectionNotation}

All graphs considered in this article are finite, simple and undirected. The vertex set and edge set of a graph $G$ are denoted respectively by $V(G)$ and $E(G)$. For a graph $G$ and any $S \subseteq V(G)$, the subgraph of $G$ induced on the vertex set $S$ is denoted by $G[S]$. For any $v \in V(G)$, we use $N_G(v)$ to denote the neighbourhood of $v$ in $G$, i.e., $N_G(v) = \{u \in V(G) : \{v,u\} \in E(G)\}$. 

A \emph{closed interval} on the real line, denoted as $[i,j]$ where $i,j \in \R$ and $i\leq j$, is the set $\{x\in \R : i\leq x\leq j\}$. Given an interval $X=[i,j]$, define $\leftend(X)=i$ and $\rightend(X)=j$. We say that the closed interval $X$ has \emph{left end-point} $\leftend(X)$ and \emph{right end-point} $\rightend(X)$. For any two intervals $[i_1, j_1], [i_2,j_2]$ on the real line, we say that $[i_1, j_1] < [i_2,j_2]$ if $j_1 < i_2$. 

For any finite positive integer $n$, we shall use $[n]$ to denote the set $\{1,\ldots , n\}$. A permutation of a finite set $V$ is a bijection from $V$ to $[|V|]$. The logarithm of any positive real number $x$ to the base $2$ and $e$ are respectively denoted by $\log(x)$ and $\ln(x)$, while $\ilog(x)$ denotes the iterated logarithm of $x$ to the base $2$, i.e. the number of times the logarithm function (to the base $2$) should be applied so that the result is less than or equal to $1$.

\section{Upper bound: $k$-degenerate graphs}
\label{sectionDegeneracy}

For any non-negative integer $n$, a \emph{star} $S_n$ is a rooted tree on $n+1$ nodes with one root and $n$ leaves connected to the root. In other words, a star is a tree with at most one vertex whose degree is not one. A \emph{star forest} is a  disjoint union of stars. 

\begin{definition}
\label{definitionArboricty}
The \emph{arboricity} of a graph $G$, denoted by $\mathcal{A}(G)$, is the minimum number of spanning forests whose union  covers all the edges of $G$. The \emph{star arboricity} of a graph $G$, denoted by $\mathcal{S}(G)$, is the minimum number of spanning star forests whose union covers all the edges of $G$. 
\end{definition}

Clearly, $\mathcal{S}(G) \geq \mathcal{A}(G)$ from definition. Furthermore, since any tree can be covered by two star forests, $\mathcal{S}(G) \leq 2\mathcal{A}(G)$. 

For the sake of completeness, we give a proof for the following already-known lemma on star arboricity of $k$-degenerate graphs (Definition \ref{definitionDegeneracy}). 

\begin{lemma}
\label{lemmaStarArboricityDegeneracy}
For a $k$-degenerate graph $G$, $\mathcal{S}(G) \leq 2k$.  
\end{lemma}
\begin{proof}
By following the degeneracy order, the edges of $G$ can be oriented acyclically such that each vertex has an out-degree at most $k$. Now the edges of $G$ can be partitioned into $k$ spanning forests by choosing a different forest for each outgoing edge from a vertex. Thus, $\mathcal{A}(G) \leq k$ and $\mathcal{S}(G) \leq 2k$.   
\end{proof}

With this we now give a proof of our first result.

\subsubsection*{Proof of Theorem \ref{theoremBoxliDegeneracy}.}

\noindent \textit{Statement.}
For a $k$-degenerate graph $G$ on $n$ vertices, $\boxli(G) \in O(k \log \log n)$.

\begin{proof}
Let $B = \{b_1, \ldots , b_n\}$ and let $r = \floor{ \log \log n + \frac{1}{2}\log \log \log n + \log(\sqrt{2} \pi) + o(1) }$.  From \cite{scramble}, we know that there exists a family $\E =\{\sigma^1, \ldots , \sigma^r\}$ of permutations of $B$ that is $3$-suitable for $B$. Recall that a family $\E$ of permutations of $[n]$ is called $3$-suitable if for every $a, b_1, b_2 \in [n]$ their exists a permutation $\sigma \in \E$ such that $\{b_1, b_2\} \prec_{\sigma} \{a\}$.

By Lemma \ref{lemmaStarArboricityDegeneracy}, we can partition the edges of $G$ into a collection of $2k$ spanning star forests. Let $\mathcal{C} = \{C_1, \ldots , C_{2k} \}$ be one such 
collection. Each star in each star forest has exactly one root vertex which is a highest degree vertex in the star (ties resolved arbitrarily). 

Consider a spanning forest $C_i$, $i \in [2k]$. We construct a family $\F_i = \{\sigma_i^1 , \ldots, \sigma_i^r , \overline{\sigma}_i^1 , \ldots, \overline{\sigma}_i^r\}$ of permutations of $V(G)$ from $C_i$ as follows. In the permutation $\sigma_i^j$, the vertices of the same star of $C_i$ come together as a block, the blocks are ordered according to the permutation $\sigma^j$; within every block the root vertex comes last; and the leaves are ordered according to $\sigma^j$. The permutation $\overline\sigma_i^j$ is similar to $\sigma_i^j$ except that the blocks are ordered in the reverse order. This is formalised in Construction \ref{constructionBoxliDegeneracy}. Let $L_i$ and $l_i$, $i \in [2k]$ be functions from $V(G) \into B$ such that the following two properties hold.

\setcounter{property}{0}
\begin{property}
\label{property1Degeneracy}
$L_i(u) = L_i(v)$ if and only if $u$ and $v$ belong to the same star in $C_i$ 
\end{property}
\begin{property}
\label{property2Degeneracy}
If $u$ and $v$ belong to the same star in $C_i$, then $l_i(u) \neq l_i(v)$. 
\end{property}

It is straight forward to construct such functions.

\begin{construction}(Constructing $\sigma_i^j$ and $\overline{\sigma}_i^j$).
\label{constructionBoxliDegeneracy}
\begin{algorithmic}
\vspace{1ex}
\STATE{For any distinct $u,v \in V(G)$, }
\IF{$L_i(u) \neq L_i(v)$} 
\STATE{/*$u$ and $v$ belong to different stars in $C_i$ */}
\STATE{$u \prec_{\sigma_i^j} v \iff L_i(u) \prec_{\sigma^j} L_i(v)$}
\STATE{$u \prec_{\overline{\sigma}_i^j} v \iff L_i(v) \prec_{\sigma^j}  L_i(u)$}
\ELSE 
\STATE{/*$u$ and $v$ belong to the same star in $C_i$ */}
\IF{$u$ is the root vertex of its star in $C_i$}
\STATE{$v \prec_{\sigma_i^j} u$}
\STATE{$v \prec_{\overline{\sigma}_i^j} u$}
\ELSIF{$v$ is the root vertex of its star in $C_i$}
\STATE{$u \prec_{\sigma_i^j} v$}
\STATE{$u \prec_{\overline{\sigma}_i^j} v$}
\ELSE
\STATE{$u \prec_{\sigma_i^j} v \iff l_i(u) \prec_{\sigma^j} l_i(v)$}
\STATE{$u \prec_{\overline{\sigma}_i^j} v \iff l_i(u) \prec_{\sigma^j} l_i(v)$}
\ENDIF
\ENDIF
\end{algorithmic}
\end{construction} 

\begin{claim}
\label{claimBoxliDegeneracy}
$\mathcal{F} =  \bigcup_{i=1}^{2k} \F_i$ is a pairwise-suitable family of permutations for $G$. 
\end{claim}

Let $\{a,b\}, \{c,d\}$ be two disjoint edges in $G$. Let $C_i$ be the star forest which contains the edge $\{a, b \}$. We will show that one of the permutations in $\F_i$ constructed above will separate these two edges. Since the edge $\{a, b\}$ is present in $C_i$ for some $i \in [2k]$, the vertices $a$ and $b$ belong to the same star, say $S$, of $C_i$ with one of them, say $a$, as the root of $S$. If the vertices $c$ and $d$ are not in $S$ then $3$-suitability among the stars (blocks) is sufficient to separate the two edges. If $c$ and $d$ are in $S$, then the $3$-suitability within the leaves of $S$ suffices. If only one of $c$ or $d$ is in $S$, then the $3$-suitability among the leaves is sufficient to realise the separation of the two edges in one of the two corresponding permutations of the blocks. The details follow.

\setcounter{case}{0}

%
%

\begin{case}[$c,d \in V(S)$]
Then by Property \ref{property1Degeneracy}, $L_i(a) = L_i(b) = L_i(c) = L_i(d)$. Since $\E = \{\sigma_1 , \ldots , \sigma_r\}$ is a $3$-suitable family of permutations for $B= \{b_1, \ldots , b_n \}$, there exists a permutation, say $\sigma^j \in \E$, such that $\{l_i(c), l_i(d)\} \prec_{\sigma^j} \{l_i(b)\}$. Then, from Construction \ref{constructionBoxliDegeneracy}, we have $\{c,d\} \prec_{\sigma_i^j} b$. Since $a$ is the root vertex of the star $S$ in $C_i$ we also have $u \prec_{\sigma_i^j} a$, for all $u \in V(S) \setminus \{a\}$. Thus, $\{c,d\} \prec_{\sigma_i^j} \{a,b\}$.
\end{case}

\begin{case}[only $c \in V(S)$] 
Then, by Property \ref{property1Degeneracy}, $L_i(a) = L_i(b) = L_i(c)$ and $L_i(c) \neq L_i(d)$. Moreover, by Property \ref{property2Degeneracy}, $l_i(a)$, $l_i(b)$ and $l_i(c)$ are distinct. Since $\E$ is a $3$-suitable family of permutations for $B$, there exists a $\sigma^j \in \E$ such that $l_i(c) \prec_{\sigma^j} l_i(b)$. Combining this with the fact that $a$ is the root vertex of $S$, using Construction \ref{constructionBoxliDegeneracy}, we get $c \prec_{\sigma_i^j} b \prec_{\sigma_i^j} a$ and $c \prec_{\overline{\sigma}_i^j} b \prec_{\overline{\sigma}_i^j} a$. Recall that $L_i(c) \neq L_i(d)$. If $L_i(d) < L_i(c)$, then we get $d \prec_{\sigma_i^j} c \prec_{\sigma_i^j} b \prec_{\sigma_i^j} a$. Otherwise, we get $d \prec_{\overline{\sigma}_i^j} c \prec_{\overline{\sigma}_i^j} b \prec_{\overline{\sigma}_i^j} a$. 
\end{case}

\begin{case}[only $d \in V(S)$] 
This is similar to the previous subcase. 
\end{case}

\begin{case}[$c,d \notin V(S)$] 
If $c$ and $d$ belong to the same star in $C_i$, say $S'$, then by Property $P_1$,  we have $L_i(a) = L_i(b)$, $L_i(c) = L_i(d)$, and $L_i(a) \neq L_i(c)$. Then for any $j \in [r]$, either $L_i(a)  \prec_{\sigma^j} L_i(c)$ or $L_i(c) \prec_{\sigma^j} L_i(a)$. Therefore, either $\{a,b\} \prec_{\sigma_i^j} \{c,d\}$ or $\{c,d\} \prec_{\sigma_i^j} \{a,b\}$. If $c$ and $d$ belong to different stars in $C_i$, then Property $P_1$ ensures that $L_i(c)$, $L_i(d)$ and $L_i(a)$ are distinct. Since $\E$ is a $3$-suitable family of permutations for $B$, there exists a $\sigma^j \in \E$ such that $\{L_i(c), L_i(d)\} \prec_{\sigma^j} L_i(a)$. This, combined with Construction \ref{constructionBoxliDegeneracy}, implies that $\{c,d\} \prec_{\sigma_i^j} \{a,b\}$. 
\end{case}

Thus, we prove Claim \ref{claimBoxliDegeneracy}. Applying the same, we get $\boxli(G) \leq |\F| = \sum_{i=1}^{2k}|\F_i| = 4kr = 4k\floor{ \log \log n + \frac{1}{2}\log \log \log n + \log(\sqrt{2} \boxli) + o(1)  }$. 
\end{proof}

%


\section{Upper bound: Fully subdivided graphs} 
\label{sectionSubdivision}

In this section we establish an upper bound for $\boxli(G^{1/2})$ in terms of $\chi(G)$, where $\chi(G)$ denotes the chromatic number of $G$. The upper bound on $\boxli(G^{1/2})$ is obtained by constructing an interval order based on $G$ of height $\chi(G) - 1$ and then showing that its poset dimension is an upper bound on $\boxli(G^{1/2})$. We need some more definitions and notation before proceeding. 

\begin{definition}[Poset dimension]
\label{definitionPosetDimension}
Let $(\mathcal{P}, \lhd)$ be a poset (partially ordered set). A {\em linear extension} $L$ of $\mathcal{P}$ is a total order which satisfies $(x \lhd y \in \mathcal{P}) \implies (x \lhd y \in L)$. A {\em realiser} of $\mathcal{P}$ is a set of linear extensions of $\mathcal{P}$, say $\mathcal{R}$, which satisfy the following condition: for any two distinct elements $x$ and $y$, $x\lhd y \in \mathcal{P}$ if and only if $x \lhd y \in L$, $\forall L \in \mathcal{R}$.  
The \emph{poset dimension} of $\mathcal{P}$, denoted by $dim(\mathcal{P})$, is the minimum integer $k$ such that there exists a realiser of $\mathcal{P}$ of cardinality $k$. 
\end{definition}

\begin{definition}[Interval dimension]
\label{definitionIntervalDimension}
A \emph{open interval} on the real line, denoted as $(a,b)$, where $a,b \in \R$ and $a < b$, is the set $\{x \in \mathbb{R} : a <  x < b\}$. For a collection $C$ of open intervals on the real line the partial order $(C, \lhd)$ defined by the relation $(a,b) \lhd (c,d)$ if $b \leq c$ in $\R$ is called the {\em interval order} corresponding to $C$. The poset dimension of this interval order $(C,\lhd)$ is called the \emph{interval dimension} of $C$ and is denoted by $\idim(C)$.
\end{definition}

The major part of our proof of Theorem \ref{theoremSubdivisionChromaticNumber} is the following lemma.

\begin{lemma}
\label{lemmaSubdivisionIntervalOrder}
For any graph $G$ and a permutation $\sigma$ of $V(G)$, let $C_{G, \sigma}$ denote the collection of open intervals $(\sigma(u), \sigma(v)), \{u,v\} \in E(G), u \prec_{\sigma} v$. Then,
$$ \boxli(G^{1/2}) \leq \min_{\sigma} \idim ( C_{G,\sigma} ) + 2,$$
where the minimisation is done over all possible permutations $\sigma$ of $V(G)$.
\end{lemma}
\begin{proof}
Let $\sigma$ be any permutation of $V(G)$. We relabel the vertices of $G$ so that $v_1 \prec_{\sigma} \cdots \prec_{\sigma} v_n$, where $n = |V(G)|$. For every edge $e = \{v_i, v_j\} \in E(G), i < j$, the new vertex in $G^{1/2}$ introduced by subdividing $e$ is denoted as $u_{ij}$. For a new vertex $u_{ij}$, its two neighbours, $v_i$ and $v_j$ will be respectively called the {\em left neighbour} and {\em right neighbour} of $u_{ij}$. We call an edge of the form $\{v_i, u_{ij}\}$ as a {\em left edge} and one of the form $\{u_{ij}, v_j\}$ as a {\em right edge}.

Let $\mathcal{R} = \{L_1, \ldots, L_d\}$ be a realiser for $(C_{G, \sigma}, \lhd)$ such that $d = \idim(C_{G,\sigma})$. For each total order $L_p, p \in [d]$, we construct a permutation $\sigma_p$ of $V(G^{1/2})$ as follows. First, the subdivided vertices are ordered from left to right as the corresponding intervals are ordered in $L_p$, i.e, $u_{ij} \prec_{\sigma_p} u_{kl} \iff (i,j) \prec_{L_p} (k,l)$. Next the original vertices are introduced into the order one by one as follows. The vertex $v_1$ is placed as the left most vertex. Once all the vertices $v_i, i < j$ are placed, we place $v_j$ at the left most possible position so that $v_{j-1} \prec_{\sigma_p} v_j$ and $u_{ij} \prec_{\sigma_p} v_j, \forall i <j$. This ensures that $v_j \prec_{\sigma_p} u_{jk}, \forall k >j$ because $u_{ij'} \prec_{\sigma_p} u_{jk}, \forall j' \leq j$ (Since $(i,j) \lhd (j,k)$). Now we construct two more permutations $\sigma_{d+1}$ and $\sigma_{d+2}$ as follows. In both of them, first the original vertices are ordered as $v_1 \prec \cdots \prec v_n$. In $\sigma_{d+1}$, the subdivided vertices are placed immediately after its left neighbour, i.e., $v_i \prec_{\sigma_{d+1}} u_{ij} \prec_{\sigma_{d+1}} v_{i+1}$ for all $\{i, j \} \in E(G)$. In $\sigma_{d+2}$, the subdivided vertices are placed immediately before its right neighbour, i.e., $v_{j-1} \prec_{\sigma_{d+2}} u_{ij} \prec_{\sigma_{d+2}} v_{j}$ for all $\{i, j \} \in E(G)$. Notice that in all the permutations so far constructed, the left (right) neighbour of every subdivided vertex is placed to its left (right).

We complete the proof by showing that $\F = \{\sigma_1, \ldots, \sigma_{d+2}\}$ is pairwise suitable for $G^{1/2}$ by analysing the following cases.  Any two disjoint left edges are separated in $\sigma_{d+1}$ and any two disjoint right edges are separated in $\sigma_{d+2}$. If $(i,j) \lhd (k,l)$, then every pair of disjoint edges among those incident on $u_{ij}$ or $u_{kl}$ are separated in every permutation in $\F$. Hence the only non-trivial case is when we have a left edge $\{v_i, u_{ij}\}$ and a right edge $\{u_{kl}, v_l\}$ such that $(i,j) \cap (k,l) \neq \emptyset$. Since $(i,j)$ and $(k,l)$ are incomparable in $(C_{G, \sigma}, \lhd)$, there exists a permutation $\sigma_p, p \in [d]$ such that $u_{ij} \prec_{\sigma_p} u_{kl}$. Since $v_i$ is before $u_{ij}$ and $v_l$ is after $u_{kl}$ in every permutation, $\sigma_p$ separates $\{v_i, u_{ij}\}$ from $\{u_{kl}, v_l\}$.  
\end{proof}

\subsubsection*{Proof of Theorem \ref{theoremSubdivisionChromaticNumber}} 

The {\em height} of a partial order is the size of a largest chain in it. It was shown by F\"{u}redi, Hajnal, R\"{o}dl and Trotter \cite{furedi1991interval} that the dimension of an interval order of height $h$ is at most $\log\log h + (\frac{1}{2} + o(1))\log\log\log h$ (see also Theorem $9.6$ in \cite{trotter1997new}). A proof of theorem \ref{theoremSubdivisionChromaticNumber} is now immediate.

\bigskip
\noindent \textit{Statement.} For a graph $G$ with chromatic number $\chi(G)$, 
$$ 
	\boxli(G^{1/2}) \leq 
		\log\log (\chi(G)-1) + 
		\left( \frac{1}{2} + o(1) \right) \log\log\log (\chi(G)-1) + 2.
$$

\begin{proof}
Let $V_1, \ldots, V_{\chi(G)}$ be the colour classes of an optimal proper colouring of $G$. Let $\sigma$ be a permutation of $V(G)$ such that $V_1 \prec_{\sigma} \cdots \prec_{\sigma} V_{\chi(G)}$. Now it is easy to see that the longest chain in $(C_{G,\sigma}, \lhd)$ is of length at most $\chi(G) - 1$. Hence the statement follows from the result of F\"{u}redi et al. \cite{furedi1991interval} and Lemma \ref{lemmaSubdivisionIntervalOrder} above.
\end{proof}



\section{Lower bound: Fully subdivided clique} 
\label{sectionSubdividedClique}

It easily follows from Theorem \ref{theoremSubdivisionChromaticNumber} that $\boxli(K_n^{1/2}) \in O(\log \log n)$. In this section we prove that $\boxli(K_n^{1/2}) \geq \frac{1}{2} \log \log (n-1)$, showing the near tightness of that upper bound. We give a brief outline of the proof below. (Definitions of the new terms are given before the formal proof.) 

First, we use Erd\H{o}s-Szekeres Theorem \cite{ErdosSzekeres} to argue that for any family $\F$ of permutations of $V(K_n^{1/2})$, with $|\F| < \frac{1}{2} \log \log n$, a subset $V'$ of original vertices of $K_n^{1/2}$, with $n' = |V'| \approx 2^{\sqrt{\log n}}$, is ordered essentially in the same way by every permutation in $\F$. Since the ordering of the vertices in $V'$ are fixed, the only way for $\F$ to realise pairwise suitability among the edges in the subdivided paths between vertices in $V'$ is to find suitable positions for the new vertices (those introduced by subdivisions) inside the fixed order of $V'$. We then show that this amounts to constructing a realiser for the canonical open interval order $(C_{n'}, \lhd)$ and hence $|\F|$, in this case, is lower bounded by the poset dimension of $(C_{n'}, \lhd)$ which is known to be at least $\log \log (n'-1) \approx \frac{1}{2} \log \log (n-1)$. 

\begin{definition}[Canonical open interval order]
\label{definitionCanonicalOpenInterval}
For a positive integer $n$, let $C_n = \{(a,b) : a, b \in [n], a < b \}$ be the collection of all the ${n \choose 2}$ open intervals which have their endpoints  in $[n]$. Then $(C_n,\lhd)$, the interval order corresponding to the collection $C_n$, is called the {\em canonical open interval order}. 
\end{definition} 

Usually the canonical interval order is defined over closed intervals.  For a positive integer $n$, let $I_{n} = \{[a,b]: a, b \in [n], a \leq b \}$ be the collection of all the ${n+1 \choose 2}$ closed intervals which have their endpoints in $[n]$. The poset $(I_n,\lhd')$, where $[i, j] \lhd' [k,l] \iff j < k$ is called the {\em canonical (closed) interval order} in literature. It is easy to see that $f: (C_n, \lhd) \into (I_{n-1}, \lhd')$, with $f((i,j)) = [i, j-1]$ is an isomorphism. It is well known that the dimension of $(I_{n-1}, \lhd')$ and hence $(C_n, \lhd)$ is at most $\log\log (n-1) + (\frac{1}{2} + o(1))\log\log\log (n-1)$. We state below the known lower bound for the same for later reference.

\begin{theorem}[F\"{u}redi, Hajnal, R\"{o}dl, Trotter \cite{furedi1991interval}]
\label{theoremIdimCanonicalLowerBound}
$$dim(C_n) \geq \log\log(n-1),$$
\end{theorem} 

\subsubsection*{Proof of Theorem \ref{theoremKnHalf}}

\noindent \textit{Statement.}
Let $K_n^{1/2}$ denote the graph obtained by fully subdividing $K_n$. Then, 
$$ \frac{1}{2} \floor{\log\log(n-1)} \leq \boxli(K_n^{1/2}) \leq (1 + o(1))\log\log (n-1).$$ 

\begin{proof}
The upper bound follows from Theorem \ref{theoremSubdivisionChromaticNumber}. So it suffices to show the lower bound. 

Let $v_1, \ldots , v_n$ denote the \textit{original vertices} (the vertices of degree $n-1$) in $K_n^{1/2}$ and let $u_{ij}$, $i, j \in [n], i < j$, denote the new vertex of degree $2$ introduced when the edge $\{i,j\}$ of $K_n$ was subdivided. Let $\F$ be a family of permutations that is pairwise suitable for $K_n^{1/2}$ such that $|\F| = r = \boxli(K_n^{1/2})$. For convenience, let us assume that $n$ is exactly one more than a power of power of $2$, i.e., $\log\log (n-1) \in \N$. The floor in the lower bound gives the necessary correction otherwise when we bring $n$ down to the largest such number below $n$. Let $p=(n-1)^{1/2^r} + 1$. 

By Erd\H{o}s-Szekeres Theorem \cite{ErdosSzekeres},  we know that if $\tau$ and $\tau'$ are two permutations of $[n^2 + 1]$, then there exists some $X \subseteq [n^2 + 1]$ with $|X|=n+1$ such that the permutations $\tau$ and $\tau'$ when restricted to $X$ are the same or reverse of each other. By repetitive application of this argument, we can see that there exists a set $X$ of $p$ original vertices of $K_n^{1/2}$ such that, for each $\sigma, \sigma' \in \F$, the permutation of $X$ obtained by restricting $\sigma$ to $X$ is the same or reverse of the permutation obtained by restricting $\sigma'$ to $X$. Without loss of generality, let $X = \{v_1, \ldots , v_p\}$ such that, for each $\sigma \in \F$, either $v_1 \prec_{\sigma} \cdots \prec_{\sigma} v_p$ or $v_p \prec_{\sigma} \cdots \prec_{\sigma} v_1$. Now we ``massage'' $\F$ to give it two nice properties without changing its cardinality or sacrificing its pairwise suitability for $K_n^{1/2}$.

Note that if a family of permutations is pairwise suitable for a graph then the family retains this property even if any of the permutations in the family is reversed. Hence we can assume the following property without loss of generality.

\setcounter{property}{0}
\begin{property}
\label{property1KnHalf}
$v_1 \prec_{\sigma} \cdots \prec_{\sigma} v_p, \forall \sigma \in \F$. 
\end{property}

Consider any $i,j \in [p], i < j$. For each $\sigma \in \F$, it is safe to assume that $v_i \prec_{\sigma} u_{ij} \prec_{\sigma} v_j$. Otherwise, we can modify the permutation $\sigma$ such that $\F$ is still a pairwise suitable family of permutations for $K_n^{1/2}$. To demonstrate this, suppose $v_i  \prec_{\sigma} v_j \prec_{\sigma} u_{ij}$. Then, we modify $\sigma$ such that $u_{ij}$ is the immediate predecessor of $v_j$. It is easy to verify that, for each pair of disjoint edges $e,f \in E(K_n^{1/2})$, if $e \prec_{\sigma} f$ or  $f \prec_{\sigma} e$ then the same holds in the modified $\sigma$ too. Similarly, if $u_{ij} \prec_{\sigma} v_i  \prec_{\sigma} v_j$ then we modify $\sigma$ such that $u_{ij}$ is the immediate successor of $v_i$. Hence we can assume the next property also without loss in generality.

\begin{property}
\label{property2KnHalf}
$v_i \prec_{\sigma} u_{ij} \prec_{\sigma} v_j, \forall i, j \in [p], i <j,~ \forall \sigma \in \F$.
\end{property}

These two properties ensure that for any two open intervals $(i,j)$ and $(k,l)$ in $C_p$ if $(i,j) \lhd (k,l)$ then $u_{ij} \prec_{\sigma} u_{kl}, \forall \sigma \in \F$. In the other case, i.e., when $(i,j) \cap (k,l) \neq \emptyset$, we make the following claim.

\begin{claim}
\label{claim1KnHalf}
Let $i,j,k,l \in [p]$ such that $(i,j) \cap (k,l) \neq \emptyset$. Then there exist $\sigma_a, \sigma_b \in \F$ such that $u_{ij} \prec_{\sigma_a} u_{kl}$ and $u_{kl} \prec_{\sigma_b} u_{ij}$.
\end{claim}


Since $(i,j) \cap (k,l) \neq \emptyset$, we have $k < j$ and $i < l$. Hence by Property \ref{property1KnHalf}, $\forall \sigma \in \F$, $v_k \prec_{\sigma} v_j$ and $v_i \prec_{\sigma} v_l$. Now we prove the claim by contradiction. If  $u_{ij} \prec_{\sigma} u_{kl}$ for every $\sigma \in \F$ then, together with the fact that $v_k \prec_{\sigma} v_j, \forall \sigma \in \F$, we see that no $\sigma \in \F$ can separate the edges $\{v_j, u_{ij}\}$ and $\{v_k, u_{kl}\}$. But this contradicts the fact that $\F$ is a pairwise suitable family of permutations for $K_n^{1/2}$.  Similarly if $u_{kl} \prec_{\sigma} u_{ij}$ for every $\sigma \in \F$ then, together with the fact that $v_i \prec_{\sigma} v_l, \forall \sigma \in \F$, we see that no $\sigma \in \F$ can separate $\{v_i, u_{ij}\}$ and $\{v_l, u_{kl}\}$. But this too contradicts the pairwise suitability of $\F$. Thus we prove Claim \ref{claim1KnHalf}.

With these two properties and the claim above, we are ready to prove the following claim.

\begin{claim}
\label{claim2KnHalf}
$|\F| \geq \idim((C_p, \lhd))$.
\end{claim}

For every $\sigma \in \F$, construct a total order $L_{\sigma}$ of $C_p$ such that $(i,j) \lhd (k,l) \in L_{\sigma} \iff u_{ij} \prec_{\sigma} u_{kl}$. By Property \ref{property1KnHalf} and Property \ref{property2KnHalf}, $L_{\sigma}$ is a linear extension of $(C_p, \lhd)$. Further, Claim \ref{claim1KnHalf} ensures that $\mathcal{R} = \{L_{\sigma}\}_{\sigma \in \F}$ is a realiser of $(C_p, \lhd)$. Hence $|\F| = |\mathcal{R}| \geq \idim((C_p, \lhd))$.


Now we are ready to show the final claim which settles the lower bound.

\begin{claim}
\label{claim3KnHalf}
$|\F| \geq \frac{1}{2}\log \log (n-1)$. 
\end{claim}

Suppose for contradiction that $|\F| = r < \frac{1}{2}\log \log (n-1)$. Then, by Claim \ref{claim2KnHalf}, $r \geq \idim((C_p, \lhd))$ where $p = (n-1)^{1/2^r} + 1 > 2^{\sqrt{\log (n-1)}} + 1$. But then, by Theorem \ref{theoremIdimCanonicalLowerBound}, we have $r \geq \log \log (p-1) > \log \log (2^{\sqrt{\log (n-1)}}) = \frac{1}{2}\log \log (n-1)$ which contradicts our starting assumption. 
\end{proof}




\end{document}